\documentclass[11pt]{article}
\usepackage[utf8]{inputenc}
\usepackage[T1]{fontenc}
\usepackage{macros}
\usepackage{amsmath}
\usepackage{tikz-cd}
\usepackage{tikz}
\usepackage{amsthm}
\usepackage{amssymb}
\usepackage{stmaryrd}
\usepackage{amsfonts}
\usepackage{graphicx}
\usepackage{caption}
\usepackage{subcaption}
\usepackage{xfrac, faktor}
\usepackage{xspace}
\usepackage[all]{xy}
\usepackage{txfonts}
\usepackage{csquotes}
\usepackage{mathtools}
\usepackage{thmtools}
\usepackage{tdsfrmath}
\usepackage{xfrac}
\usepackage[shortlabels]{enumitem}
\usepackage{geometry}
\usepackage{arcs}
\usepackage{txfonts}
\usepackage{comment}

\newtheorem{thm}{Theorem}[section]
\newtheorem{lem}[thm]{Lemma}
\newtheorem{cor}[thm]{Corollary}
\newtheorem{theorem}{Theorem}

\theoremstyle{remark}
\newtheorem{rem}{Remark} [section]

\newtheorem{ppt}{Property}

\theoremstyle{remark}

\theoremstyle{definition}
\newtheorem{defi}[thm]{Definition}

\newtheorem*{thm*}{Theorem}
\newtheorem*{cor*}{Corollary}
\newtheorem*{conj*}{Conjecture}

\newcommand{\Span}[1]{\operatorname{span}\set{#1}}

\DeclareFontFamily{OMX}{yhex}{}
\DeclareFontShape{OMX}{yhex}{m}{n}{<->yhcmex10}{}
\DeclareSymbolFont{yhlargesymbols}{OMX}{yhex}{m}{n}
\DeclareMathAccent{\wip}{\mathord}{yhlargesymbols}{"F3}

\title{The arc complexes of partially decorated hyperbolic polygons}
\author{Pallavi Panda}
\date{}

\begin{document}
	
	\maketitle
\paragraph{Abstract.} We consider two families of hyperbolic polygons: ideal and ideal once-punctured, some of whose spikes are decorated with horoballs. We show that the arc complexes of these two families of surfaces, generated by edge-to-edge arcs and edge-to-decorated-spike arcs, are closed piecewise linear balls. This is proved in a completely combinatorial setting: compact polygons whose vertices are assigned red or blue colouring. In order to prove the ballness, we show that these simplicial complexes are pseudo-manifolds and use shellability to conclude. As a consequence, we parametrise weakly-lengthening deformations of the partially decorated hyperbolic polygons.

\section{Introduction}
In this paper, we study the topology of a finite combinatorial object associated to two types of hyperbolic surfaces: decorated ideal polygons and decorated once-punctured polygons.

An ideal $n$-gon, denoted by $\ip n$ is the convex hull in $\HP$ of $n\geq 3$ points in the boundary of $\HP$. This surface is homeomorphic to a disk with $n$ points removed from its boundary. The missing points are called \emph{spikes} and the infinite hyperbolic geodesic joining two consecutive spikes is called an \emph{edge}. A once-punctured $n$-gon, denoted by $\puncp n$, with ($n\geq 1$), is obtained from an ideal ($n+2$)-gon by identifying two consecutive edges. The resulting surface is homeomorphic to a punctured disk with $n$ points removed from its boundary. 

We will refer to both these families as hyperbolic polygons. A hyperbolic polygon with $n$ spikes is said to be \emph{partially decorated} if $0\leq r\leq n$ of its spikes are equipped with horoballs. When $r=0$, the polygon is said to be \emph{undecorated}. When $r=n$, the polygon is said to be \emph{fully decorated}. We denote by $\Pi_{n,r}$ (resp. $\Pi^\times_{n,r}$) the family of ideal $n$-gons (resp. once-punctured $n$-gons) with $r$ decorated spikes. This construction is motivated by Penner's decorated Teichmuller theory, where he studied general fully decorated hyperbolic surfaces. 

The \emph{arc complex} $\hac \Pi$ of such a partially decorated surface $\Pi$ is a pure clique (simplicial) complex using certain embedded arcs and their intersectionality. This was first introduced by Harvey \cite{Harvey} for general hyperbolic surfaces with boundary. In our case, the 0-simplices are given by the isotopy classes of two types of arcs: edge-to-edge arcs, and decorated-spike-to-edge arcs. Two such 0-simplices are joined by a 1-simplex if there exists representatives in their isotopy classes that are disjoint. The arc family corresponding to a maximal simplex is called a \emph{triangulation} of the polygon. One can construct a flip graph whose vertices are given by triangulations; two vertices are joined by an edge if it is possible to obtain one triangulation from the other by exchanging only one arc. 

Penner \cite{penner} proved that the arc complexes of undecorated ideal polygons as well as undecorated once-punctured polygons are piecewise linear spheres. In general, arc complexes are locally non-compact with infinite diameter. The two families of hyperbolic polygons are amongst a short list of exceptions. The flip graph of undecorated $\ip n$ is the 1-skeleton of the famous polytope called \emph{associahedron}, discovered by Tamari \cite{tamari}, and then rediscovered by Stasheff \cite{stash}. Sleator--Tarjan--Thurston \cite{tarjan} and Pournin \cite{diameterpournin} have determined the diameter of this polytope using hyperbolic and combinatorial methods. The flip graph of the undecorated $\puncp n$ is the 1-skeleton of the polytope named \emph{cyclohedron} or type-B associahedron, introduced by Bott-Taubes. It is also the flip graph of centrally symmetric triangulations of an ideal polygon. See \cite{PanGor21} for the work of Panina--Gordon on the triangulations of this surface.

The main result of this paper is to give the topology of the arc complexes of partially decorated polygons. We prove that
\begin{theorem} \label{thm: A}
	For $n\geq 3$ and $2\leq r\leq n$, the arc complex $\hac{\Pi}$ of a partially decorated ideal polygon $\Pi\in\pdep nr$ is $PL$-homeomorphic to a closed ball of dimension $n+r-4$.
\end{theorem}

\begin{theorem} \label{thm: B}
	For $n\geq 1$ and $1\leq r\leq n$, the arc complex $\hac{\Pi}$ of a partially decorated once-punctured polygon $\Pi\in\pdepu nr$ is $PL$-homeomorphic to a closed ball of dimension $n+r-2$.
\end{theorem}

The motivation behind this topological result comes from hyperbolic geometry. A \emph{horoball connection} is the geodesic segment subtended by the horoball decorations of two spikes. The length of this segment measures the distance between the two decorated spikes. The (lambda) lengths of the horoball connections were used by Penner to give a cellular decomposition of the decorated Teichmuller space of a surface. An \emph{admissible} deformation of a partially decorated polygon with at least one horoball connection is the set of all infinitesimal deformations of the hyperbolic metric that uniformly lengthens all horoball connections. This set forms an convex polyhedral cone, denoted by $\adm m$. In \cite{panda2023strip}, we prove that the cone over the arc complex of decorated ideal polygons as well as decorated once-punctured polygons, parametrises this admissible cone. 

This work was motivated by the work of Danciger--Guéritaud--Kassel \cite{dgk}, who proved a similar result in the case of convex cocompact hyperbolic surfaces. The admissible cone is formed by the set of all deformations that uniformly lengthen every closed geodesic. To a weighted sum of pairwise disjoint arcs, they assigned an infinitesimal deformation obtained by gluing strips along these arcs. This map is called the \emph{strip map}. Such a deformation is admissible. Strip deformations were first defined by Thurston \cite{thurston}. Goldman--Labourie--Margulis, in \cite{glm}, proved that the admissible cone forms an open convex cone. Moreover, a hyperbolic metric along with an admissible deformation corresponds to a flat Lorentzian non-compact 3-manifold called a \emph{Margulis spacetime}, whose fundamental group is given by a finitely generated free group. 
We generalised these results to general decorated hyperbolic surfaces \cite{panda2024dmst}. An admissible deformation in this case is an infinitesimal deformations that uniformly lengthens every horoball connection. In \cite{panda2024dmst}, we showed that the second definition implied the first one. Finally, we also gave a parametrisation of Margulis spacetimes, decorated by affine lightlike lines. 

The boundary of the admissible cone consists of weakly-lengthening deformations. It is very difficult to study these deformations using strip deformation because the polytope has infinitely many facets and the full arc complex is not a PL-manifold in general. Nevertheless, we take advantage of the small surfaces like decorated polygons that have finite ball-like arc complexes in order to get a better insight of the weakly lengthening deformations of a general surface. In particular, we get a parametrisation of the boundary using strip deformations:

\begin{theorem}
	Let $\Pi$ be a partially decorated (possibly once-punctured) polygon with a metric $m\in \tei \Pi$. Then the strip map $f:C \hac \Pi\longrightarrow \tang \Pi$ is a homeomorphism onto $\overline{\adm m}$.
\end{theorem}

The proofs of our main theorems \ref{thm: A},\ref{thm: B}, are done in a completely combinatorial and topological setting. To a partially decorated (resp. once-punctured) hyperbolic polygon $\Pi\in \pdep nt$ one can associate a compact (resp. once-punctured) polygon $\mathcal P$ with $n+r$ vertices, $r$ of which are coloured red and the rest are colored blue. The red vertices correspond to the decorated spikes and the blue vertices correspond to the boundary edges of the hyperbolic polygon. Then the subcomplex $\cac$ of the arc complex $\mathcal A(\mathcal{P})$, generated by $R-B$ arcs and $B-B$ arcs is isomorphic to the corresponding hyperbolic polygon. In fact, we prove something stronger:
\begin{theorem}
	The subcomplex $\cac$ of a polygon $\poly m$ with any $R-B$ bicolouring is a closed ball of dimension $m-4$. Similarly, the subcomplex $\cacp$ of a once-punctured polygon $\punc m$ with any bicolouring is a closed ball of dimension $m-2$.
\end{theorem}
The proof of our main theorems are motivated by the work of Wilson \cite{wilson}. He proved that the arc and curve complex of a Möbius strip with marked points on the boundary is spherical. He also reproved the sphericity of the arc complexes of a convex polygon and a once-punctured polygon. The main ingredient he used was the shellability of pure simplicial complexes. This is an ordering of all the maximal simplices of the complex so that the intersection of the $k$-th simplex with the union of the first $(k-1)$ simplices is always a pure complex of codimension one. He gave shelling orders for the arc complexes of all three surfaces and then concluded using a result by Danaraj and Klee \cite{danaraj} which states that a shellable $d$-pseudo-manifold without boundary is a combinatorial sphere.
We follow a similar approach.

As a final corollary, we have that,

\begin{theorem}
	Let $\Pi$ be a partially decorated $n$-gon with $r$ decorated spikes and $h\leq r$ marked boundary horoball connections. Then, the arc complex $\hac {\Pi}$ is PL-homeomorphic to a closed ball of dimension $n+r-h-4$ (resp. $n+r-h-2$) , for $\Pi\in\pdep nr$ (resp. $ \pdepu nr$), except when $h=r=n$, in which case there are no permitted arcs. 
\end{theorem}

This helps us parametrise the maximal dimensional facets of the admissible cone using strip deformations. A facet on the boundary of the admissible cone corresponds to all those infinitesimal deformations that lengthen every horoball connection except one. 

\paragraph{Organisation of the paper.}The paper is structured into two parts in the following way:
The first consists of combinatorial and topological proofs. The second part discusses the application in hyperbolic geometry. Section  \ref{setup} recapitulates the necessary vocabulary and results from simplicial complexes, arc complexes and hyperbolic geometry. Section \ref{sectionmain} contains the proofs of our main theorems. Section \ref{appli} describes the link between the arc complexes of the Euclidean polygons and the decorated hyperbolic polygons.
\paragraph{Acknowledgements.}  
This work was partially done at Université de Luxembourg supported by Luxembourg National Research Fund OPEN grant O19/13865598. It was completed at Université Sorbonne Paris Nord, CNRS, Laboratoire d'Informatique de Paris Nord, LIPN. I would like to thank my supervisors Fran\c cois Gu\'eritaud, Hugo Parlier and Lionel Pournin for helpful discussions and encouragement.

\section{Setup}\label{setup}
\subsection{Simplicial complex}
In this section we recall relevant definitions and results on finite simplicial complexes.

A simplicial complex is called \emph{pure} if all of its maximal simplices have the same dimension.
The \emph{dual graph} of a simplicial complex is the graph whose vertices are the maximal simplices and two vertices are joined by an edge if the corresponding maximal simplices share a codimension 1 face. A pure simplicial complex is said to be \emph{strongly connected} if its dual graph is connected.
\vspace{0.5em}

	A \emph{$d$-pseudo-manifold with boundary} is a pure strongly connected $d$-simplicial complex in which every $(d-1)$-simplex is contained in atmost two $d$-simplices. Note that the boundary of such a simplicial complex is formed by all $(d-1)$-simplices that are contained in exactly one $d$-simplex.
\vspace{0.5em}

Next we recall the definition of a shelling of pure simplicial complexes.

	Let $X$ be a pure finite simplicial complex of dimension $d$. A \emph{shelling} of $X$ is an enumeration of its maximal simplices $\mathcal{T}:( C_1,\ldots,C_n)$ such that for every $1\leq k\leq n$, the intersection $\pa{\bigcup\limits_{j=1}^{k-1}C_j} \bigcap C_k$ is a pure simplicial complex of dimension $d-1$.

The following is a lemma linking shellability and join of two simplicial complexes that we shall use in the proof of our main theorems. See \cite{joinshelling} for a proof.
\begin{lem}\label{join}
	Two complexes $X,Y$ are shellable if and only if $X\Join Y$ is shellable.
\end{lem}

We will use the following result by Danaraj and Klee \cite{danaraj} to prove our main theorems.
\begin{thm}\label{topo}
	A shellable $d$-pseudo-manifold with boundary is $PL$-homeomorphic to a closed ball of dimension $d$.
\end{thm}

\subsection{Arcs and arc complexes}
In this section we recall the relevant definitions and results the arc complex of polygons that will be used in the rest of the paper.

We denote by $S_{g,n}$ a surface with genus $g\,(\geq0)$ and $n\,(\geq 1)$ marked points on its boundary. 
\begin{defi}\label{defiarc}
The arc complex $\mathcal{A}(S_{g,n})$ of a finite-type surface $S_{g,n}$ with marked points on its boundary is a simplicial complex defined in the following way: the 0-skeleton is given by the embedded arcs with their endpoints on the marked points of the polygon, up to homotopy relative to the endpoints. For $k\geq 1$, every $k$ simplex is given by a $(k+1)$-tuple of pairwise disjoint and distinct arcs, up to homotopy.
\end{defi}
In this section, we are going to consider only two types of surfaces: convex polygons $\poly m$ $(m\geq 4)$ and once-punctured convex polygons $\punc m$, $(m\geq 2)$.
In the case of a convex polygon, these homotopy classes are simply the diagonals joining two vertices of the polgon. To avoid confusion, we will refer to the arcs of a punctured polygon as diagonals as well. A \emph{maximal} diagonal of $\punc m$ is the diagonal with both its endpoints coinciding on a vertex of the polygon. The blue diagonal in the left panel of Fig. \eqref{acpuncsq} is a maximal diagonal.

 \begin{figure}
 	\begin{subfigure}{\linewidth}
 		\centering
 		\includegraphics[width=12cm]{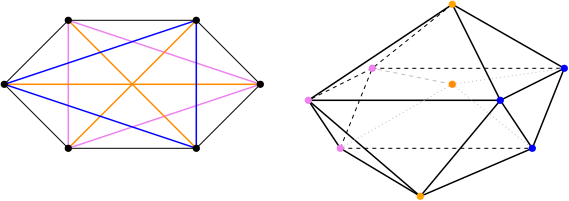}
 		\subcaption{The diagonals of a hexagon $\poly6$ and $\ac[6]$}
 		\label{achexa}
 	\end{subfigure}
 	\vspace{0.8cm}
 	
 	\begin{subfigure}{\linewidth}
 		\centering
 		\includegraphics[width=12cm]{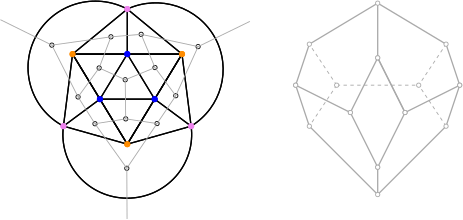}
 		\subcaption{The flattened perspective of $\ac[6] $ and its dual associahedron of dimension 2}
 		\label{asso}
 	\end{subfigure}
 	\caption{}
 \end{figure}

 The following is a classical fact from combinatorics.  See, for instance, \cite{penner} for a proof by Penner. 
 \begin{thm}
The arc complex of a convex polygon $\poly m$ ($m\geq 4$) is PL-homemorphic to a sphere of dimension $m-4$. 
 \end{thm}
 Fig.\eqref{achexa} shows the diagonals and the arc complex of a hexagon. The diagonals corresponding to the 0-skeleton of a maximal simplex of $\ac$ decomposes the polygon $\poly m$ into triangles. In the case of a punctured polygon $\punc m$, the diagonals decompose it into triangles and a once-punctured disc with one marked point on its boundary. Hence a maximal simplex will be alternatively referred to as a \emph{triangulation} of the polygon. A triangulation, all of whose diagonals are incident on the same vertex, is called a \emph{fan} triangulation.
 
 The dual graph to $\mathcal{A}(\mathcal{P})$ is called a \emph{flip graph}.
  In the case of a convex polygon $\poly n$, the dual graph is the 1-skeleton of ($n-4$)-dimensional associahedron, which we mentioned in the introduction. See Fig.\eqref{asso} for the associahedron of dimension 3.
 
 The following theorem about the arc complex of once-punctured polygons was proved by Penner in \cite{penner}.
\begin{thm}\label{idpac}~
	The arc complex $\acp$ of a punctured $m$-gon, ($m\geq 2$), is PL-homeomorphic to a sphere of dimension $n-2$.
\end{thm}
Fig.\ref{acpuncsq} gives an illustration of the arc complex of a once-punctured quadrilateral. The blue diagonal in the left panel is a maximal arc.
\begin{figure}[!ht]
	\centering
	\includegraphics[width=10cm]{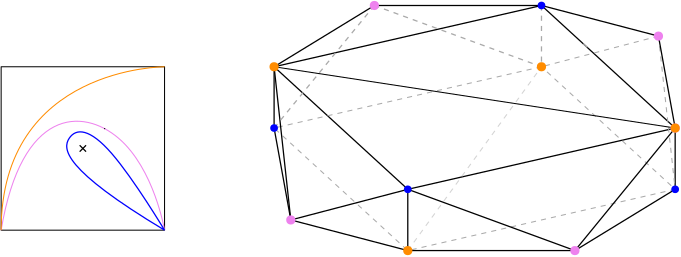}
	\caption{The three types of diagonals and the full arc complex of $\punc{4}$}
	\label{acpuncsq}
\end{figure}

\paragraph{Bicolourings.}\label{bicolouring}
Given a convex (possibly punctured) polygon we consider all possible colourings of its vertices with two colours, say red $(R)$ and blue $(B)$, so that there is a vertex of each colour. Such a colouring is called a \emph{bicolouring}. A bicolouring is called \emph{non-trivial} if there is at least one $R-R$ diagonal. 
\begin{figure}[h]
	\centering
	\includegraphics[width=12cm]{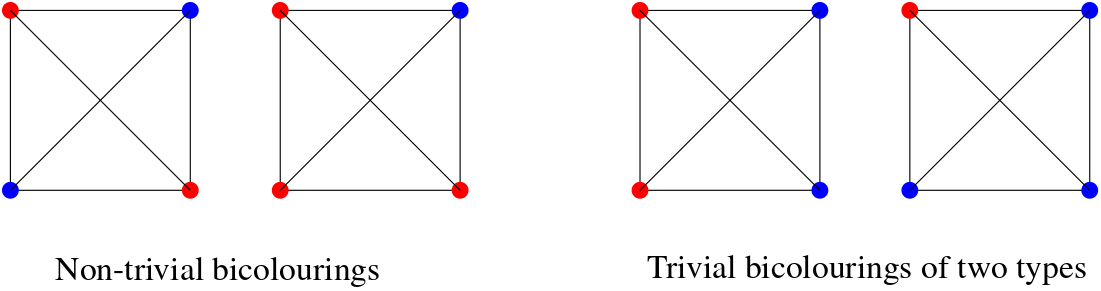}
	\caption{}
	\label{bicol}
\end{figure}

See Fig.\eqref{bicol} for examples of trivial and non-trivial bicolourings of a quadrilateral $\poly 4$.


We denote by $\cac$  (resp. $\cacp$) the subcomplex of $\ac$ (resp. $\acp$) generated by the $R-B$ and $B-B$ diagonals only. We call these diagonals \emph{permitted} and the $R-R$ diagonals are called \emph{rejected}. A simplex of $\cac$ using permitted diagonals is called \emph{permissible}. 
\begin{rem}
Any fan triangulation of a bicoloured polygon based at a blue vertex is permissible.
\end{rem}
\begin{rem}
In the case of a trivial bicolouring, the subcomplex is the full arc complex because there are no rejected diagonals. This is why the bicolouring is named "trivial".
\end{rem}
\begin{figure}[!ht]
	\centering
	\includegraphics[width=12cm]{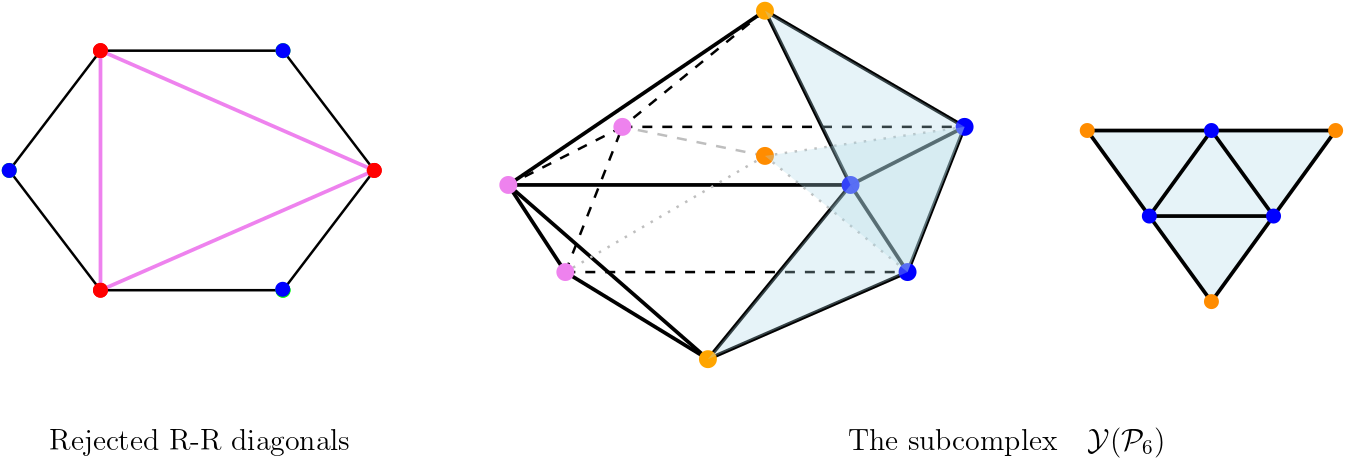}
	\caption{}
	\label{achexaalt}
\end{figure}
\begin{figure}[!ht]
	\centering
	\includegraphics[width=12cm]{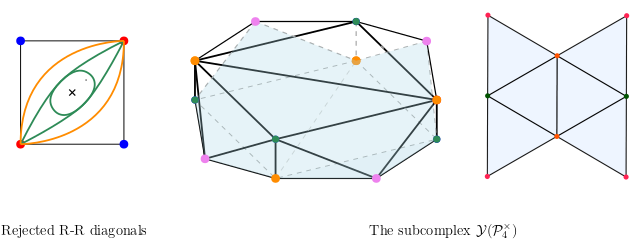}
	\caption{}
	\label{acsqalt}
\end{figure}

Figs. \eqref{achexaalt} and \eqref{acsqalt} show the subcomplexes $\cac[6]$ and $\cacp[4]$ when the polygons $\poly 6,\, \punc 4$ have alternate bicolouring. In both the cases, they are closed 2-balls.

\subsection{Shelling of finite arc complexes}
Wilson gave an equivalent definition for shelling in \cite{wilson}, in the context of arc complexes.
\begin{ppt}[Wilson]\label{wilson}
	Let $X$ be a pure simplicial complex. There exists an enumeration of all the maximal faces of $X$, denoted by $\mathcal{T}: (C_1,\ldots,C_n)$, such that for every two positive integers $j$ and $k$ satisfying $j<{k}\leq n$, there exists a positive integer $i<k$ such that
	\begin{enumerate}
		\item[(i)]\label{p1} $C_i\cap C_k$ is a codimension one simplex,
		\item[(ii)] \label{p2}$C_j\cap C_k\subset C_i\cap C_k$.
	\end{enumerate}
\end{ppt}
In \cite{wilson} (Proposition 3.5), Wilson proved that in the context of finite arc complexes, there is an equivalence between shellability and Property \eqref{wilson}. 

\begin{lem}
Suppose that $S_{g,n}$ is a surface with finite arc complex. Then the complex  $\mathcal{A}(S_{g,n})$ is shellable if and only if it satisfies Property \eqref{wilson}. 
\end{lem}
In the following Lemma we show that the same holds for any pure subcomplex of the arc complex.
\begin{lem}\label{equiv}
	A pure codimension zero subcomplex $X$ of $\mathcal{A}(S_{g,n})$ is shellable if and only if it satisfies Property \eqref{wilson}.
\end{lem}
\begin{proof}
	We show that the enumeration $\mathcal{T}=C_1\ldots,C_n$ given by the property works as a shelling order. So we need to show that for every $1\leq k\leq n$, the simplicial complex $B_k:=\pa{\bigcup\limits_{j=1}^{k-1}C_j }\bigcap C_k$ is a pure simplicial complex of dimension $d-1$, where $d$ is the dimension of the subcomplex $X$ as well as $\mathcal{A}(S_{g,n})$. Let $\sigma \subset B_k$ be any simplex of codimension more than one. Then it is contained in $C_j\cap C_k$ for some $j\in \{1,\ldots,k-1\}$. Since $X$ is a subcomplex of the arc complex, the complex $C_j\cap C_k$ is in fact a simplex. From Property \eqref{wilson} we get that there is an $i<k$ such that $C_i\cap C_k$ is of dimension $d-1$ and $C_j\cap C_k\subset C_i\cap C_k\subset B_k$. Hence we get that the simplex $\sigma$ is not maximal in $B_k$ and the dimension of any maximal simplex containing $\sigma$ is $d-1$. 
	
	Conversely, let us suppose that $X$ is shellable with a shelling order $\mathcal{T}_n=(T_1,\ldots,T_n)$. Consider two maximal simplices $T_j, T_k$ with $j<k\leq n$. If $T_j\cap T_k=\varnothing$, there is nothing to show. So we assume that $T_j\cap T_k\neq \varnothing$. Once again, the intersection $T_j\cap T_k$ is a simplex because $X$ is a subcomplex of the arc complex. Let $B_k:= \pa{\bigcup\limits_{i=1}^{k-1} T_i }\cap T_k$ as before. From the hypothesis, $B_k$ is a pure simplicial complex of dimension $d-1$. So we get that $T_j\cap T_k$ is contained in a $(d-1)$-simplex $T_i\cap T_k$, where $T_i$ is a maximal simplex with $i<k$. This $T_i$ satisfies the two conditions of Property \eqref{wilson}.
\end{proof}
The orientable surfaces with finite arc complexes are convex polygons $\poly m$, once-punctured polygons $\punc m$, annulus with one marked point in one boundary component and $m$ marked points on the other boundary component. Note that a once-punctured convex polygon is an annulus with marked points only on one boundary component. In \cite{palesi}, Dupont and Palesi introduce the \emph{quasi} arc complex of a non-orientable surface where they include any one-sided curve in the 0-skeleton. The only non-orientable surface with finite (quasi) arc complex is a Möbius strip with marked points on its boundary. Wilson proved the sphericity of this complex in \cite{wilson}, using shellability. In order to prove the shellability, he first gave shelling orders of the arc complexes of a convex polygon and a once-punctured polygon. 
\begin{thm}[Wilson, Proposition 2.11, Claim 1.4]\label{ac}
	The arc complex $\ac$ of a convex polygon $\poly m$ ($m\geq 4$) is a $(m-4)$-pseudo-manifold with boundary and is shellable.
\end{thm}

\begin{thm}[Wilson, Proposition 2.11, 3.13]\label{acp}
	The arc complex $\acp$ of a punctured polygon $\punc m$ ($m\geq 2$) is a $(m-2)$-pseudo-manifold with boundary and is shellable.
\end{thm}

\section{Main theorems}\label{sectionmain}
\subsection{Bicoloured convex polygon}
In this section we prove that the subcomplex $\cac \subset \ac$ is PL-homeomorphic to  a closed ball. Firstly we show that this subcomplex is a pseudo-manifold with boundary when the bicolouring is non-trivial.

\begin{lem}\label{pure}
For $m\geq 4$ and any non-trivial bicolouring of $\poly m$, the subcomplex $\cac$ of $\ac$ is pure of dimension $m-4$.
\end{lem}
\begin{proof}
Let $\sigma$ be a simplex of $\cac$ of dimension $k\geq 1$. The diagonals corresponding to the 0-simplices of $\sigma$ decompose the polygon $\poly m$ into $k+2$ regions, some of which are untriangulated polygons with at least 4 vertices. Since these diagonals are permitted, every smaller untriangulated polygon has a blue vertex. We triangulate each one of them with a fan triangulation based at one such blue vertex. Since the dimension of the maximal simplex in the uncoloured polygon is $m-4$, we get that any maximal simplex containing $\sigma$ is of dimension $m-4$. 
\end{proof}
Next, we show that the subcomplex is strongly connected. 
\begin{figure}[!ht]
	\centering
	\includegraphics[width=12cm]{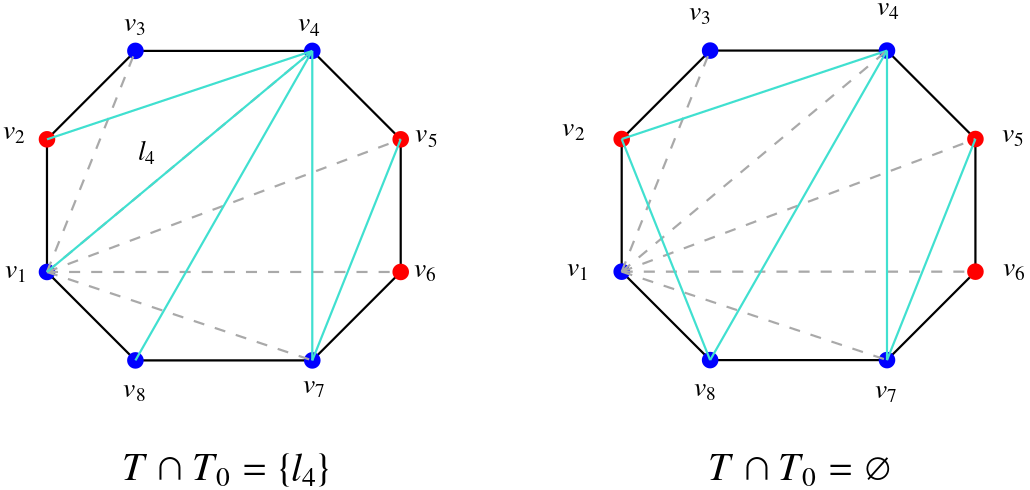}
	\caption{}
	\label{strongly}
\end{figure}
\begin{lem}\label{strong}
For $m\geq 4$ and any non-trivial bicolouring, the subcomplex $\cac$ is strongly connected. 
\end{lem}
\begin{proof}
We will prove that $T$ is connected by flips to the fan triangulation based at any blue vertex, by induction on $m$. For $m=4$, there are two non-trivial bicolourings possible: either alternating $R,B$ vertices or three $R$ vertices. In each case the subcomplex $\cac$ is a single 0-simplex which is also the fan triangulation based simultaneously at the two blue vertices. Suppose that for any bicolouring of polygons with $1,\ldots, m-1$ vertices, any triangulation $T$ is connected by flips to a fan triangulation. Now, consider the polygon $\poly m$. We name $v_1$ any blue vertex and we enumerate the rest of the vertices in the clockwise direction. Let $T_0$ be the fan triangulation based at $v_1$. 

Suppose that $l_i\in T\cap T_0$ for some $i=3,\ldots , m-2$. See the left panel of Fig.\eqref{strong}. It divides the polygon $\poly m$ into two smaller polygons $\poly {i}, \poly{m-i+1}$ whose vertices are $v_1,v_2,\ldots, v_i$ and $v_i, v_{i+1}, \ldots, v_1$, respectively. The triangulation $T$ triangulates these two polygons. It is possible that that one of these smaller polygons have a trivial coloring. In this case, we use fact that the dual of the full arc complex (the associahedron) is connected. In particular, the restriction of $T$ to this smaller polygon is connected by flips to the fan triangulation based at a blue vertex. If both the smaller polygons have a non-trivial bicolouring, we use our induction hypothesis inside each of them to connect $T$ by flips to the fan triangulation based at $v_1$. So $T$, as a triangulation of $\poly m$, is connected by flips to the triangulation $T_0$. 

Now we suppose that there are no diagonals in $T$ that are incident at $v_1$ and since $T$ is a permissible triangulation, the vertex $v_1$ is contained in a unique triangle with vertices at $v_2, v_m$ such that at least one is blue. The diagonal joining $v_2,v_m$ is the edge of exactly one other triangle, whose third vertex is say $v_j\neq v_1$. Let $T'$ be the triangulation obtained by flipping the diagonal joining $v_2,v_m$ with the permissible diagonal $l_j$ joining $v_1, v_j$. Now we are in the first case. This concludes the induction step.

\end{proof}

Next we show that 
\begin{lem}\label{mfd}
	Every codimension one simplex of the subcomplex $\cac$ is contained in at most two maximal simplices of $\cac$. 
\end{lem}

\begin{proof}
	Let $\sigma$ be a $d-1$-simplex, where $d$ is the dimension of $\cac$. Since the complex $\ac$ of an uncoloured polygon is a pseudo-manifold, we have that $\sigma$ is contained in exactly two $d$-simplices of $\ac$. These two simplices represent two triangulations which have all but one diagonal in common. These two diagonals can be either $R-B$ and $B-B$ or both $B-B$. Indeed, if both were $R-R$ diagonals or both were $R-B$ diagonals, then there would be a $R-R$ diagonal inside $\sigma$ bounding the quadrilateral region where the two intersecting diagonals lie. This is impossible because $\sigma\subset \cac$. When both the diagonals are of type $B-B$, then both the maximal simplices lie inside $\cac$. Otherwise only the maximal simplex generated by $R-B, B-B$ diagonals lies in $\cac$. In this case, $\sigma$ is a boundary simplex of $\cac$.
	\end{proof}
Using Lemmas \eqref{pure}-\eqref{mfd}, we get that
\begin{thm}\label{pseudo}
The subcomplex $\cac$ is a $d$-pseudo-manifold with boundary.
\end{thm}

As the last step, we prove that the subcomplex is shellable.
\begin{thm}\label{propcac}
	For $m\geq 4$, the simplicial complex $\cac$ satisfies Property \eqref{wilson}.
\end{thm}
\begin{proof}
	We use induction on the number of vertices. For $m=4$, there are two possible colourings. In both the cases, the complex is a single 0-simplex which satisfies the property trivially. Let us now suppose that the statement is true for $4,\ldots, m-1$. 
	Due to the bicolouring, there is a edge of the polygon joining one red (the right one, say) and one blue vertex (left one). Enumerate the vertices in the anticlockwise direction starting from the blue vertex so that the red vertex is numbered $m$. See Fig.\ref{shell1}.
	\begin{figure}
			\centering
		\includegraphics[width=5cm]{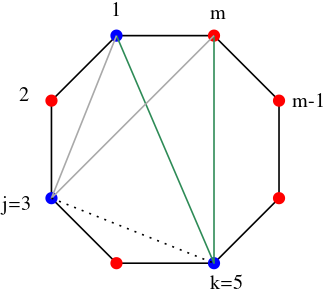}
		\caption{}
		\label{shell1}
	\end{figure}
	For $i=3,\ldots, m-2$, let $\sigma_i$ be the simplex of $\cac$ generated by the two diagonals joining the vertex $i$ to the vertices 1 and $m$. In Fig. \eqref{shell1}, the diagonals of $\sigma_3$ are drawn in grey and those of $\sigma_5$ are in greem. If the vertex $2$ is a blue vertex, then let $\sigma_2$ be the 0-simplices corresponding to the diagonal joining $(2,m)$. Let $\sigma_{m-1}$ be the 0-simplices corresponding to the diagonals joining $(1,m-1)$, which is either $R-B$ or $B-B$ because the vertex 1 is a blue vertex.  For each $i$, the simplex $\sigma_i$ decomposes the polygon $\poly m$ into an $i$-gon, an $(m-i+1)$-gon and a triangle with vertices at $1,m,i$. So the star of the simplex $\sigma_i$, $\str{\sigma_i}$, in the complex $\cac$ is the join of $\cac[i]$ and $\cac[m-i+2]$. 
	
	From the induction hypothesis, there exist enumerations of the maximal simplices of  $\cac[i]$ and $\cac[m-i+2]$, respectively, satisfying Property \eqref{wilson}. From Lemma \eqref{equiv}, the complexes $\cac[i], \cac[m-i+1]$ are shellable. From Lemma \eqref{join}, we get that $\str{\sigma_i}$ is shellable. Again from Lemma \eqref{equiv}, we get an enumeration $\mathcal{T}^i$ of the maximal simplices of $\str{\sigma_i}$. Let $J\subset \{2,\ldots,m-1\}$ be the blue vertices. Consider the ordering $\mathcal{T}:=\mathcal T^{m-1}, \mathcal T^{i_1}1,\ldots, \mathcal T^{i_p},$ where $i_1>i_2>\ldots>i_p$ are all the elements of $J$. We claim that $\mathcal{T}$ is an ordering of all the maximal simplices of $\cac$ satisfying the conditions of Property\eqref{wilson}.
	
	First we show that every triangulation appears in $\mathcal{T}^i$  for exactly one $i\in \{3,m-2\}$. 
	Let $T$ be any triangulation of $\poly m$ generated by $R-B$ and $R-R$ diagonals. Then either there is no diagonal of $T$ incident at the vertex 1, or there is at least one, in which case the last diagonal of $T$ (anti clockwise direction) must join 1 to a blue vertex, say $k\in \{2,m-1\}$, of the polygon. In Fig. \eqref{shell1}, we have that the last blue vertex joined to 1 is at 5. Since this diagonal is the last one, there must be a diagonal of $T$ joining the vertices $k$ and $m$.

	Hence $T\in \mathcal{T}^j$. Also, for every $j\neq j'$, we have $\mathcal{T}^j\cap \mathcal{T}^{j'}=\varnothing$ because the diagonal ($j,m$) intersects the diagonal ($1,j'$). 
	
	Now suppose that $S,T$ are two maximal simplices of $\cac$ such that $S$ precedes $T$. We need to find another maximal simplex $P$ preceding $T$ such that $P$ is obtained by flipping a diagonal of $T$ such that any diagonal common to $S$ and $T$ must be common to $P$ and $T$.  If $S,T$ belong to the same $\mathcal T_i$, then we can conclude immediately by using the induction hypothesis. So we assume that $S\in \mathcal T_j$ and $T\in  \mathcal T_k$ with $j>k$. Any diagonal common to both the triangulations $S$ and $T$ must lie outside the quadrilateral $1,k,j,m$ because they must be disjoint from the diagonals $(1,k)$, $(k,j)$, $(j,m)$. See Fig.\ref{shell1}.
Now let $P$ be the triangulation obtained from $T$ by flipping the diagonal $(m,k)$ by $(1,j)$. Since this operation happens inside the quadrilateral, we get that $S\cap T \subset P\cap T$. Also by construction, $P\in \str{\sigma_j}$ which precedes $\str{\sigma_k}$ and hence $T$. 
\end{proof}

From Theorem \ref{topo}, Lemma \ref{pseudo} and Theorem \ref{propcac}, we get that

\begin{thm}\label{main2}
For $m\geq 4$, the subcomplex $\cac$ of a convex polygon $\poly m$ with a non-trivial bicolouring is a $PL$-homeomorphic closed ball of dimension $m-4$.
\end{thm}
\subsection{Punctured bicoloured polygons}
In this section, we prove that the subcomplex $\cacp\subset\acp$ is a closed ball of dimension $m-2$ for any non-trivial bicolouring. First we show that the subcomplex is a pseudo-manifold with boundary. 

\begin{lem}\label{purep}
	For $m\geq 2$ and any non-trivial bicolouring, the subcomplex $\cacp$ of $\acp$ is pure of dimension $m-2$.
\end{lem}
\begin{proof}
Let $\sigma$ be any simplex of $\cacp$. If one of its 0-simplices correspond to a maximal diagonal, one of the regions in its complement is a bicoloured unpunctured polygon. See Fig.\eqref{cutmax}.
\begin{figure}
	\centering
	\includegraphics[width=10cm]{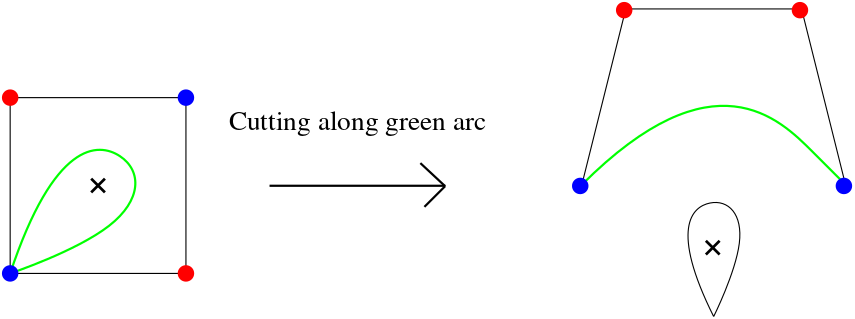}
	\caption{Complementary regions to a maximal diagonal}
	\label{cutmax}
\end{figure}

Using Lemma \ref{pure}, we have that $\sigma$ is contained in some permissible maximal simplex of $\cacp$. If none of the 0-simplices represent a maximal diagonal, let $l\in \sigma$ be the diagonal that separates the puncture from the rest of the diagonals of $\sigma$. At least one of the endpoints of $l$ is a blue vertex. We add to $\sigma$ the maximal diagonal based at this vertex. Again use Lemma \ref{pure} on its complement to get a maximal simplex of $\cacp$ containing $\sigma$.
\end{proof}
Next we show that the subcomplex is strongly connected.
\begin{lem}\label{strongp}
	For $m\geq 2$ and any non-trivial bicolouring, the subcomplex $\cacp$ is strongly connected.
\end{lem}
\begin{proof}
Let $T$ be a permissible triangulation of $\punc m$. Then there is a maximal diagonal based at a blue vertex. Cutting along the diagonal we get a bicoloured unpunctured polgygon, triangulated by the rest of the diagonals of $T$. Let $T_0$ be the fan triangulation based at one of the two blue vertices given by the maximal diagonal. By Lemma \ref{strong}, $T$ is connected to $T_0$ by flips and all the intermediate triangulations remain disjoint from the maximal diagonal.  
\end{proof}

Now we show that
\begin{lem}\label{mfdp}
	Every codimension one simplex of the subcomplex $\cacp$ is contained in at most two maximal simplices of $\cacp$. 
\end{lem}
\begin{proof}
Let $\sigma$ be a codimension one simplex of $\cacp$. From Theorem \ref{acp} we get that $\sigma$ is contained in exactly two codimension 0 simplices $\eta, \eta'$ of $\acp$. From Lemma \ref{purep}, we know that at least one of $\eta, \eta'$ is inside $\cacp$. If none of the 0-simplices of $\eta, \eta'$ contain a rejected diagonal, then they both are inside $\cacp$, otherwise only one of them is.
\end{proof}

Finally we prove that the the subcomplex is shellable. 
 \begin{thm}\label{shellp}
	For $m\geq 2$, the simplicial complex $\cacp$ satisfies Property \eqref{wilson}.
 \end{thm}
\begin{proof}
	Like in the proof of Theorem \ref{propcac}, we enumerate the vertices in the anticlockwise direction so that the first vertex is blue and the $m$-th is red. Let $d_{i_1},\ldots, d_{i_p}$ be the maximal diagonals based at the blue vertices $i_1,\ldots, i_p$, respectively. The star $\str{d_{i_j}}$ is the coloured arc complex $\cac[m+1]$ of the unpunctured bicoloured polygon $\poly {m+1}$ obtained cutting $\punc m$ along $d_{i_j}$. Let $\mathcal{T}^i$ be the shelling order given by Theorem \ref{propcac}. We claim that $\mathcal{T}:=\mathcal{T}^{i_1},\ldots, \mathcal{T}^{i_p}$ is a shelling order of $\cacp$ satisfying Property \ref{wilson}.
	Since every triangulation of $\punc m $ contains exactly one maximal diagonal based at some blue vertex, $\mathcal{T}$ is an enumeration of the maximal simplices of $\cacp$.
	
	Now suppose that $S,T$ are two maximal simplices of $\cacp$ such that $S$ precedes $T$. If they belong to the same star $\str{d_{i_j}}$, then we conclude using the proof of Theorem \ref{propcac}. Now let $S\in \str{d_{i_j}}$ and $T\in \str{d_{i_k}}$ with $j<k\leq p$. The diagonals of $S\cap T$, if any, lie outside the region containing $d_{i_k}\cup d_{i_k}$ bounded by the two diagonals joining the vertices $i_j, i_k$. See Fig.\eqref{maxswap}.
	\begin{figure}
		\centering
		\includegraphics[width=5cm]{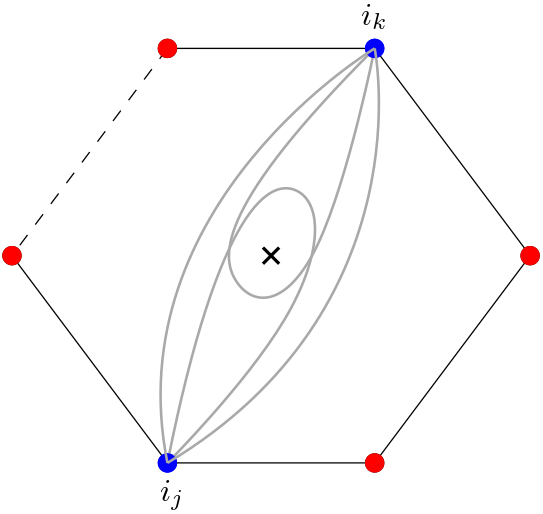}
		\caption{}
		\label{maxswap}
	\end{figure}
	Let $P$ be the triangulation obtained from $T$ by flipping the diagonal $d_{i_k}$ with $d_{i_j}$. So $P\in \str{d_{i_j}}$, hence preceding $T$ in the order. Also by construction of $P$, we have that $S\cap T \subset P\cap T$. This concludes the proof. 
\end{proof}
	\begin{figure}[!ht]
	\centering
	\includegraphics[width=10cm]{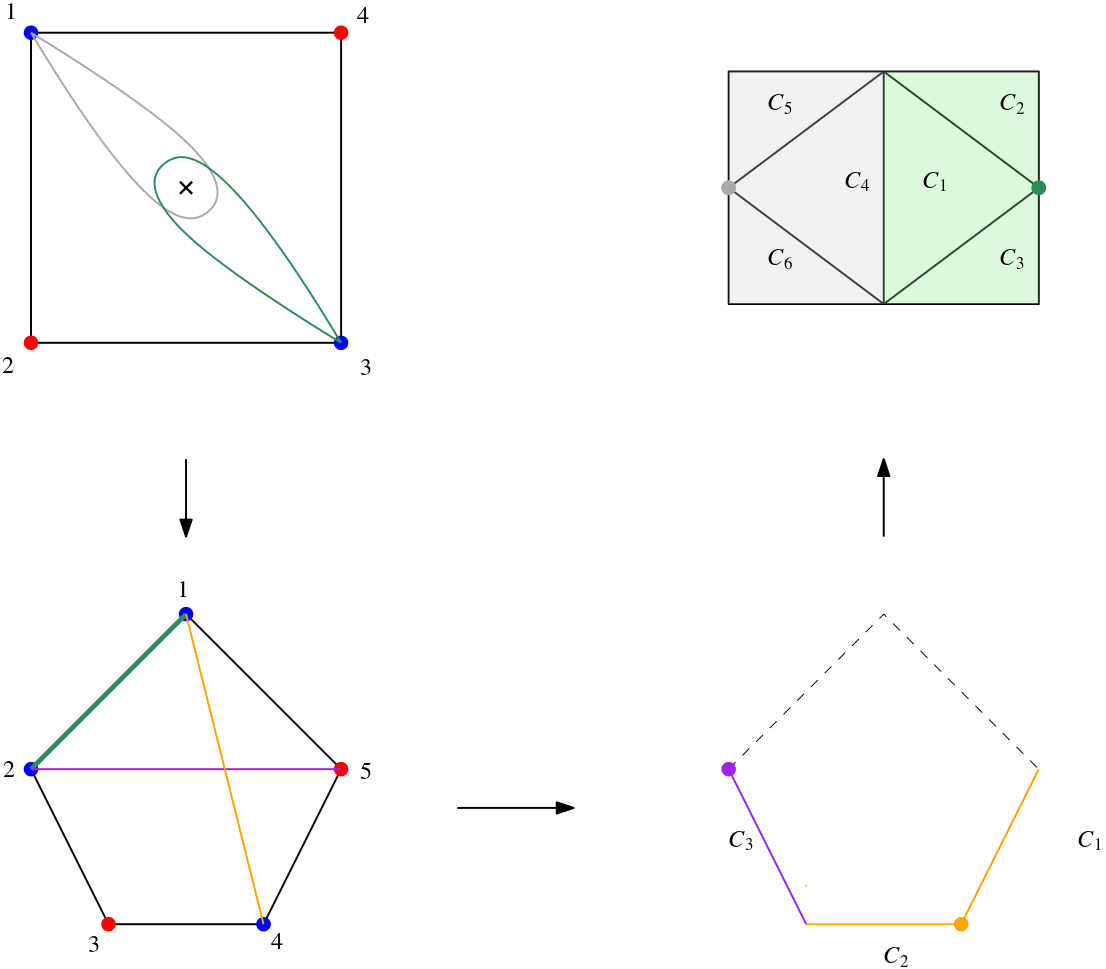}
	\caption{A shelling order for $\cacp[4]$ with alternate bicolouring}
	\label{shellorder}
\end{figure}
See Fig.\eqref{shellorder} for a shelling order for the subcomplex $\cacp[4]$ when $\punc 4$ is endowed with an alternate colouring. In this case, there are only two permissible maximal diagonals: $d_1$ and $d_3$. Cutting along $d_1$, we get an unpunctured pentagon with a non-trivial bicolouring. As in the proof of Theorem \eqref{main2}, we consider the stars of $\sigma_2$ (purple) and $\sigma_4$ (orange). The subcomplex $\cac[5]$ is a closed 1-ball as shown in the bottom right panel. A shelling order of this subcomplex gives a shelling order for $\str{d_2}$. Similarly, we get a shelling order for $\str{d_3}$. Combining the two, we get a shelling order for $\cacp[4]$. 

From Lemmas \eqref{purep},\eqref{strongp},\eqref{mfdp} and Theorem \eqref{mainp}, we get that,
\begin{thm}\label{mainp}
	For $m\geq 2$, the subcomplex $\cacp$ of a once-punctured polygon $\punc m$ with any non-trivial bicolouring is a closed ball of dimension $m-2$.
\end{thm}
%
%

\section{Applications: decorated hyperbolic polygons}\label{appli}
\subsection{Hyperbolic geometry}
In this section we recall a few definitions from hyperbolic geometry which will be used in Section \ref{appli} to give the application of our main theorems in the context of hyperbolic polygons, as previewed in the introduction. 

The hyperbolic plane, denoted by $\HP$, is the unique (up to isometry) complete simply-connected Riemannian 2-manifold of constant curvature equal to -1. Its boundary, denoted by $\HPb$, is homeomorphic to a circle $\s 1$. 
The subset $\{z=x+iy\in \C \,| \,y>0\}$ of the complex plane is the upper half-space model of the hyperbolic space of dimension 2.
The geodesics are given by semi-circles whose centres lie on $\R$ or straight lines that are perpendicular to $\R$. We shall call the former as \emph{horizontal} and the latter as $vertical$ geodesics. The boundary at infinity $\HPb$ is given by $\R \cup \{\infty\}$. The orientation-preserving isometry group is given by $\psl$ that acts by Möbius transformations on $\HP$. This group has three types of elements: elliptic (one fixed point in $\HP$), parabolic (one fixed point in $\HPb$) and hyperbolic (two fixed points in $\HPb$).


\paragraph{Horoball connection.}A horocycle based at a point in $\HPb$ is the orbit of a parabolic element fixing that point. A horoball is the convex hull of a horocycle. A geodesic, whose endpoints are decorated with pairwise disjoint horoballs, is called a \emph{horoball connection}. The length of a horoball connection is the hyperbolic length of the geodesic segment subtended by the two horoballs.

\subsection{The polygons}

\paragraph{Ideal polygons.}An \emph{ideal $n$-gon}, denoted by $\ip n$, is defined as the convex hull in $\HP$ of $n\,(\geq3)$ distinct points (called \emph{vertices}) on $\HPb$. The \emph{edges} are infinite geodesics of $\HP$ joining two consecutive vertices. Two consecutive edges meet at a point on $\HPb$, forming a \emph{spike}. The restriction of the hyperbolic metric to an ideal polygon gives it a metrically complete finite-area hyperbolic metric with geodesic boundary.

\paragraph{Punctured polygons.}For $n\geq 2$, an \emph{ideal once-punctured $n$-gon}, denoted by $\puncp n$, is another non-compact hyperbolic surface with geodesic boundary, obtained from an ideal $(n+2)$-gon, by identifying two consecutive edges using a parabolic element of $\psl$ that fixes the common vertex. The edges of the polygon are the connected components of the boundary. The vertices are the quotients of the vertices of $\ip{n+2}$.

\begin{figure}[h!]
	\centering
	\includegraphics[width=10cm]{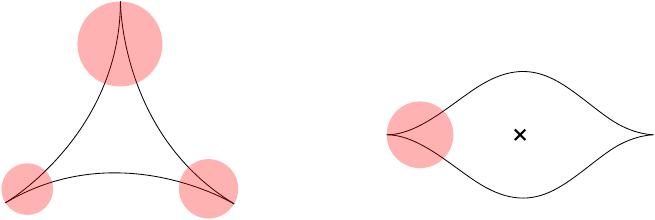}
	\caption{The two types of decorated hyperbolic polygons}
	\label{4typespoly}
\end{figure}
A spike at a vertex $v$ of an ideal (possibly punctured) polygon is said to be \emph{decorated} if a horoball, based at $v$ is added. 
\paragraph{Decorated Polygons.} A \emph{partially decorated hyperbolic polygon} is a hyperbolic polygon with some of its spikes decorated with horoballs. For $n\geq 3$ and $2\leq r\leq n$, we denote by $\pdep{n}{r}$, the set of all ideal $n$-gons with $r$ decorated spikes. We similarly define $\pdepu nr$ for $n\geq 1$ and $1\leq r\leq n$. See Fig.\ref{4typespoly} for a decorated triangle and a decorated once-punctured bigon.

The deformation space of a partially decorated ideal polygon $\Pi\in\pdep n r$, ($n\geq 3$) (resp. once-punctured polygon $\Pi\in\pdepu nr$, ($n\geq 1$)) is an open ball of dimension $n+r-3$ (resp. $n+r-1$). See \cite{panda2023strip} for a proof. 


\paragraph{Admissible deformations.}Let $\mathcal{H}$ be the set of all horoball connections of the (possibly once-punctured) polygon $\Pi$. Then we can define the following smooth positive function for every $\be\in \mathcal{H}$:
\[
\begin{array}{ccl}
	l_\be:&\tei \Pi\longrightarrow& \R_{>0}\\
	&m\mapsto&\text{length of $l_\be$ w.r.t $m$.}
\end{array}
\] 
Given a hyperbolic polygon $\Pi$ with metric $m$, a vector in the tangent space $\tang{\Pi}$ is called an \emph{infinitesimal deformation} of $\Pi$.
When $\mathcal{H}\neq \varphi$, then an infinitesimal deformation $v$  is said to be \emph{admissible} if there exists a constant $K>0$ such that for all $\be\in \mathcal{H}$, \[ \mathrm{d}l_\be (v) \geq K l_\be . \] The \emph{admissible cone}, denoted by $\adm{m}$ of a metric $m\in\tei {\Pi}$ on a decorated polygonal surface $\dep n$ is defined to be the set of all admissible deformations of the given metric. The admissible cone is open and convex in $\tang {\pdep nr}$.
In the case of a decorated polygon $\Pi\in \pdep nr$, there are only finitely many horoball connections, say $\be_1,\ldots, \be_{p(n)}$. Then we have that the admissible cone is the intersection of finitely many open halfspaces:
\[\adm m=\bigcap\limits_{i=1}^{p(n)} \{dl_{\be_i}>0\}. \]

\subsection{Arc complexes of decorated polygons}
On these partially decorated polygons we consider two types of arcs.

An \emph{arc} on a hyperbolic polygon $\Pi$, is an embedding $\al$ of a closed interval $I\subset \R$ into $\Pi$. There are two possibilities depending on the nature of the interval:
\begin{enumerate}
	\item $I=[a,b]$: In this case, the arc $\al$ is finite. We consider those finite arcs that verifiy:  $\al(a),\al(b) \in \partial \Pi$ and $\al(I)\cap S=\set{\al(a),\al(b)}$. Furthermore we do not consider the finite arcs that separate only one undecorated spike from the rest of the surface.
	\item $I=[a,\infty)$: These are embeddings of hyperbolic geodesic rays in the interior of the polygon such that  $\al(a)\in \partial \Pi$.  The infinite end converges to a decorated spike of the polygon.
\end{enumerate}  


Let $\mathscr A$ be the set of all non-trivial arcs of the two types above. 
The \emph{arc complex} of a partially decorated hyperbolic is a simplicial complex $\hac{\Pi}$ whose 0-simplices are given by the isotopy classes of arcs in $ \mathscr A$ fixing the boundary and spikes, and for $k\geq 1$, every $k$-simplex is given by a $(k+1)$-tuple of pairwise disjoint and distinct isotopy classes. A simplex $\sigma$ is said to be \emph{filling} if the arcs corresponding to $\sigma^{(0)}$  decompose the surface into topological disks with at most one vertex and a punctured disk with no vertex. The \emph{pruned arc complex} of a polygon $\Pi$, denoted by $\sac \Pi$ is the union of the interiors of the filling simplices of the arc complex $\hac \Pi$.

Now we establish a link between these decorated polygons and abstract polygons with bicolourings. 
\begin{figure}[!h]
	\centering
	\includegraphics[width=12cm]{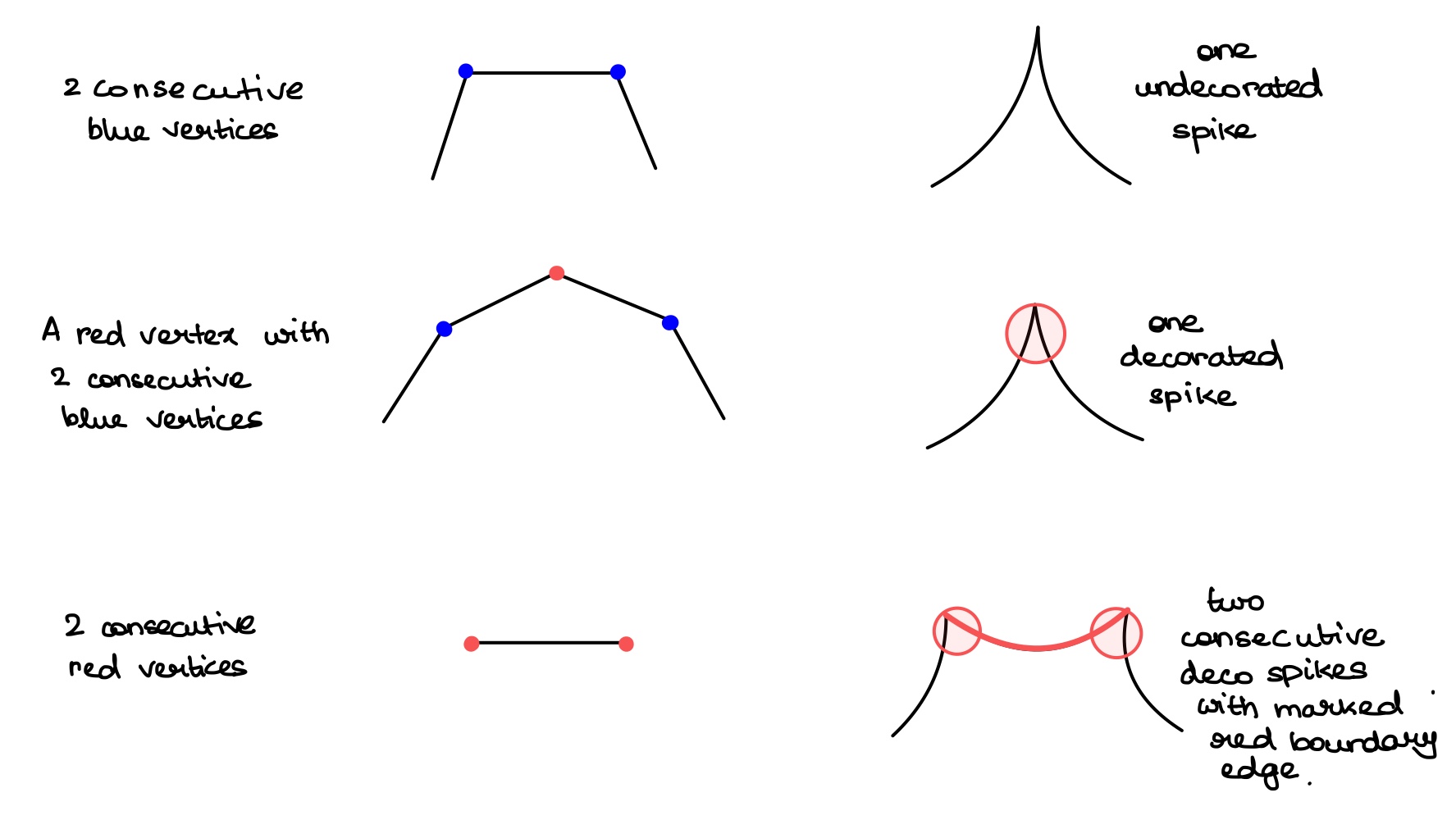}
	\caption{Correspondance between bicoloured polygons and partially decorated hyperbolic polygons.}
	\label{fig: corr}
\end{figure}
To every partially (possibly once-punctured) decorated polygon $\Pi$, one can associate the abstract polygon $\mathcal{P}$ equipped with a $B,R$-colouring of its vertices, in the following way:
\begin{itemize}
	\item a decorated vertex of $\Pi$ corresponds to a red vertex of $\mathcal{P}$,
	\item an edge of $\Pi$ corresponds to a blue vertex of $\mathcal{P}$,
	\item an undecorated spike of $\Pi$  corresponds to two consecutive blue vertices of $\mathcal{P}$,
	\item a decorated spike of $\Pi$ corresponds to one red vertex with two blue vertices on either side,
\end{itemize}
See Fig.\eqref{fig: corr}.
So we have the bijection:
\[\cur{\text{Isotopy classes of finite arcs of } \Pi} \leftrightarrow \cur{B-B\text{ diagonals}}\]
\[\cur{\text{Isotopy classes of ray-type arcs } \Pi} \leftrightarrow \cur{B-R\text{ diagonals} }\]

Hence the arc complex $\hac {\pdep nr}$ (resp. $\hac {\pdepu nr}$) is isomorphic to the subcomplex $\acsub{\poly{n+r}}$ (resp. $\acsub{\punc {n+r}}$) of $\ac[n+r]$ (resp. $\acp[n+r]$), where $\poly{n+r}$ and $\punc {n+r}$ are bicoloured polygons with $r$ red vertices and $n$ blue vertices.

From Theorem \ref{main2} we get that
\begin{cor}
For $n\geq 3$ and $2\leq r\leq n$, the arc complex $\hac{\pdep nr}$ of a decorated ideal polygon $\pdep nr$ is $PL$-homeomorphic to a closed ball of dimension $n+r-4$.
\end{cor}

Finally, from Theorem \ref{mainp}, we get that 
\begin{cor}
For $n\geq 1$ and $1\leq r\leq n$, the arc complex $\hac{\pdepu nr}$ of a decorated once-punctured polygon $\pdepu nr$ is $PL$-homeomorphic to a closed ball of dimension $n+r-2$.
\end{cor}

Now we consider a partially decorated hyperbolic polygons with marked edge-like horoball connections. Then the arc complex of is generated by arcs of $\pdep nr$ that do not touch marked boundary. This condition translates to:
\begin{itemize}
\item two consecutive decorated spikes with a marked edge-type horoball connection between them corresponds to two consecutive red vertices of $\mathcal{P}$,
\end{itemize}
Then we have that 
\begin{thm}
Let $\Pi\in\pdep nr, \pdepu nr$ be a partially decorated polygon with $h\leq r$ marked horoball connections. Then, the arc complexes $\hac {\pdep nr}, \hac{\pdepu nr}$ are PL-homeomorphic to closed balls of dimensions $n+r-h-4$ and $n+r-h-2$ respectively, except when $h=r=n$, in which case there are no permitted arcs. 
\end{thm}


\subsection{The strip map}
A strip deformation of a polygon is done by cutting the surface along a geodesic arc and gluing a strip of the hyperbolic plane $\HP$, without any shearing. In particular, along a finite arc we glue a hyperbolic strip and along a ray-type arc we glue a parabolic strip. 

Given a polygon $\Pi$ endowed with a metric $m\in \tei \Pi$, a strip template is the following data:
\begin{itemize}
	\item an $m$-geodesic representative $\al_g$ from every isotopy class $\al$ of arcs in $\mathcal{K}$, along which the strip deformation is performed,
	\item a point $p_\al\in \al_g$ where the waist of the strip being glued must lie.  
\end{itemize} 
For more detailed discussion, see \cite{panda2023strip}. 

Given an isotopy class of arcs $\al$ of a polygon $\Pi$ and a strip template $\{(\al_g,\pal,\wal)\}_{\al\in \mathcal{K}}$ adapted to the nature of $\al$ for every $m\in \tei {\Pi}$, 
	define the \emph{infinitesimal strip deformation}
	
	\[
	\begin{array}{cc}
		f_{\al}:&\tei {\Pi} \longrightarrow T\tei{\Pi}\\
		&m\mapsto  \frac{dF_{t\al} (m)}{dt}|_{t=0}
	\end{array}
	\]
	where $F_{t\al} (m)$ is obtained from $m$ by strip deforming along $\al$ with a fixed waist $\pal$ and the width as $t\wal$. 

	The \emph{infinitesimal strip map} is defined as:
	\[
	\begin{array}{ccrcl}
		\mathbb{P}f& : &	\hac \Pi & \longrightarrow & \ptan {\Pi}\\
		& &\sum\limits_{i=1}^{\dim \tei \Pi} c_i \al_i&\mapsto&\bra{\sum\limits_{i=1}^{\dim \tei \Pi}c_i f_{\al_i}(m)} 
	\end{array}
	\]

\begin{figure}
	\centering
\includegraphics[width=\linewidth]{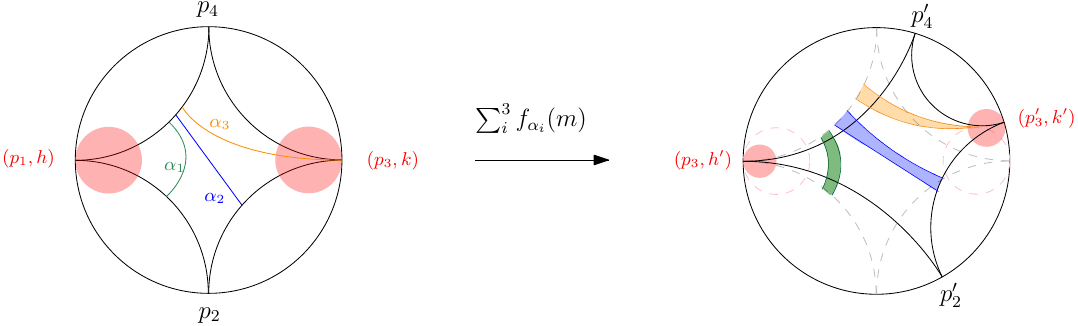}
\caption{The strip map}
\label{fig: stripmap}
\end{figure}
\vspace{0.5cm}
 See Fig. \ref{fig: stripmap} for an example.

In \cite{panda2023strip}, we proved the following parametrisation of the projectivised admissible cone.
for two types of decorated polygons — fully decorated polygons and fully decorated  once-punctured polygons. 

 \begin{thm}\label{thm: oldppdeco}
	Let $\Pi$ be a fully decorated $n$-gon $\dep n$ ($n\geq 3$)   (resp. a decorated once-punctured $n$-gon $\depu n$, for $n\geq 1$) with a decorated metric $m$  in the deformation space $\tei {\Pi}$. Fix a choice of strip template. Then the projectivised infinitesimal strip map $\mathbb{P}f: \hac \Pi\longrightarrow \ptan \Pi$, when restricted to the pruned arc complex $\sac {\Pi}$, is a homeomorphism onto its image $\mathbb{P}^+(\adm m)$ $\simeq \R^{2n-4}\, (\text{resp. } \R^{2n-2}$), where $\adm m$ is the admissible cone.
\end{thm}

We also proved a version of the above theorem for the undecorated ideal polygons and once-punctured ideal polygons.  In these cases the arcs are finite with endpoints on non-consecutive edges of the polygon. We show that the arc complex parametrises the entire positively projectivised deformation space in these cases.
\begin{thm}\label{thm: oldppundeco}
	Let $\Pi$ be an ideal polygon $\ip n$ ($n\geq4$) (resp. a once punctured polygon $\punc n$ ($n\geq2$)) with a metric $m\in \tei \Pi$. Fix a choice of strip template. Then, the projectivised infinitesimal strip map $\mathbb{P}f: \hac \Pi\longrightarrow \ptan \Pi$ $\simeq \s{n-4}\,( \text{resp. } \s{n-2})$ is a homeomorphism.
\end{thm}		

Combining these two results we get that
 
 \begin{thm}\label{cor: pardeco}
	Let $\Pi$ be a partially decorated $n$-gon $\pdep nr$ ($n\geq 3$)   (resp. a decorated once-punctured $n$-gon $\pdepu nr$ ($n\geq 1$) with a decorated metric $m\in \tei {\Pi}$. Then the projectivised infinitesimal strip map $\mathbb{P}f: \hac \Pi\longrightarrow \ptan \Pi$, when restricted to the pruned arc complex $\sac {\Pi}$, is a homeomorphism onto its image $\mathbb{P}^+(\adm m)$ $\simeq \R^{n+r-4}\, (\text{resp. } \R^{n+r-2}$), where $\adm m$ is the admissible cone.
 \end{thm}

We end this section by stating a theorem, proved in \cite{panda2023strip}, that will be useful in the next section. 
\begin{thm} \label{oldthm: basis}
	Let $\Pi\in\pdep nr$ or $\pdepu nr$ be a partially decorated hyperbolic polygon and $m\in\tei {\Pi}$ be a metric. Fix a choice of strip template. Let $\sigma$ be a top-dimensional simplex of its arc complex $\hac{\Pi}$ and let $\ed$ be the corresponding edge set. Then the set of infinitesimal strip deformations $B=\{\isd\mid e\in \ed\}$ forms a basis of the tangent space $\tang{\Pi}$.
\end{thm}

\subsection{Extension of the strip map}
In this section, we extend the strip map to the full arc complex and prove that it is a homeomorphism onto the closure of the (projectivised) admissible cone.

\begin{thm}
Let $\Pi$ be a partially decorated (possibly once-punctured) polygon with a metric $m\in \tei \Pi$. Then the projectivised strip map $\mathbb{P}f: \hac \Pi\longrightarrow \ptan \Pi$ is a homeomorphism onto $\mathbb{P}^+(\overline{\adm m})$.
\end{thm}
Before giving the proof, we illustrate the theorem with three examples. 
\begin{figure}[!h]
	\centering
	\includegraphics[width=12cm]{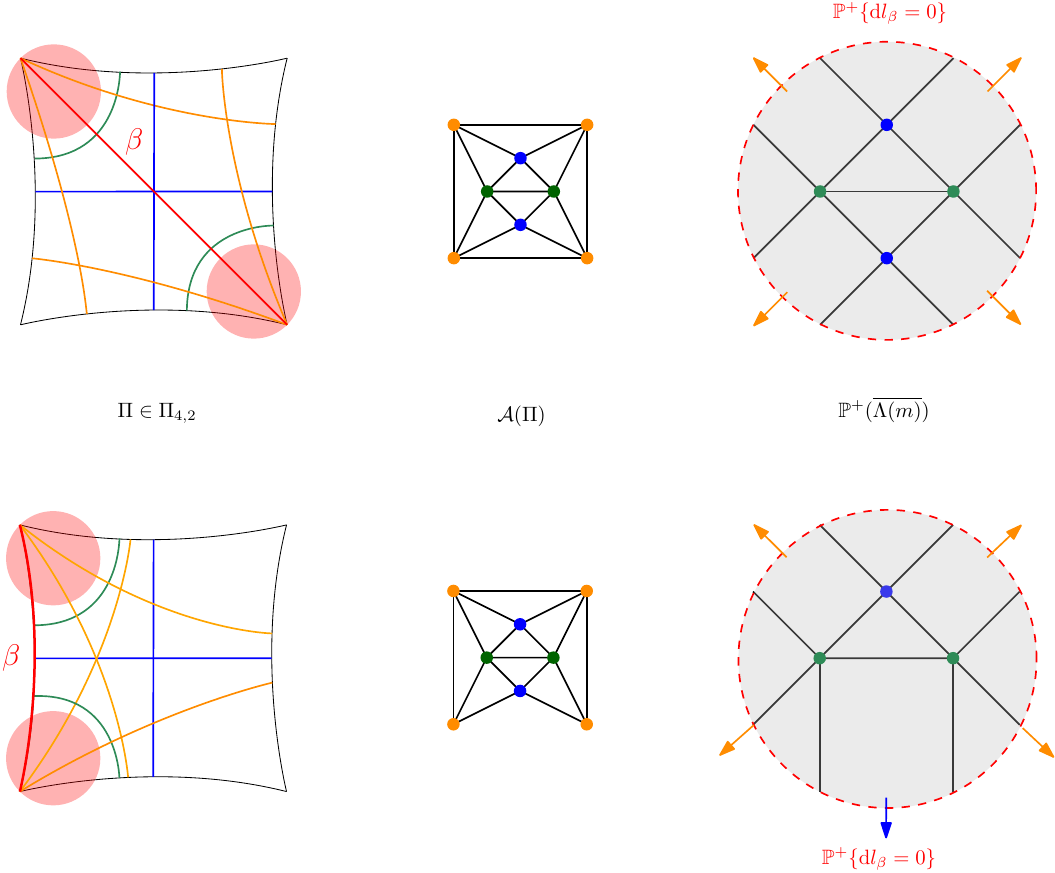}
	\caption{Arcs and horoball connections of $\depu 2$}
	\label{fig: partdeco}
\end{figure}
\begin{figure}[!h]
	\begin{subfigure}{\linewidth}
		\centering
		\includegraphics[width=8cm]{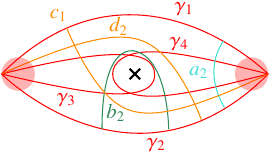}
		\caption{Arcs and horoball connections of $\depu 2$}
		\label{fig: arc2p}
	\end{subfigure}
	
	\begin{subfigure}[!h]{\linewidth}
		\centering
		\includegraphics[width=16cm]{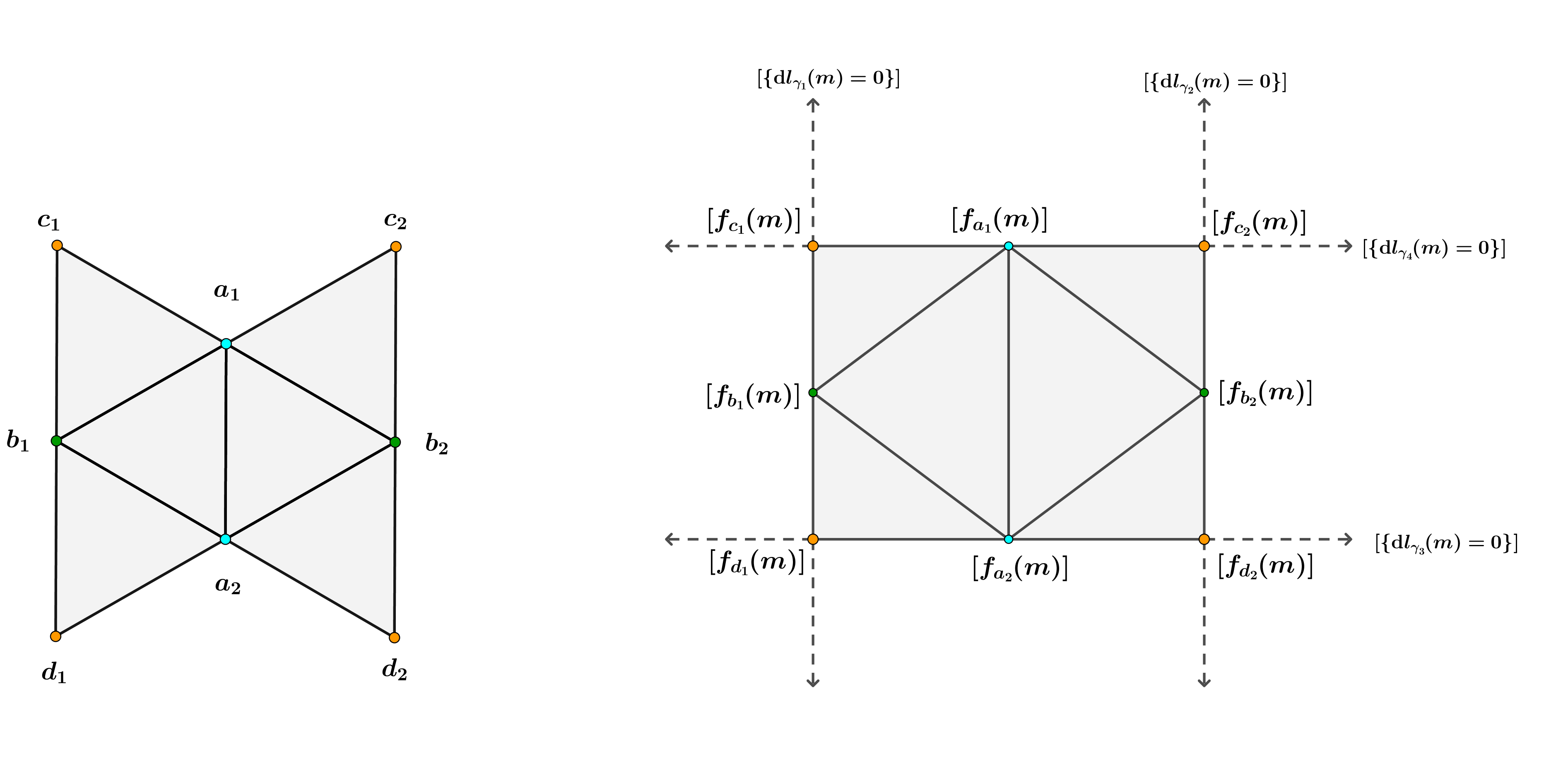}
		\caption{The projectivised strip map for $\depu 2$}
		\label{fig: strip2p}
	\end{subfigure}
\end{figure}
These are shown in Figures \ref{fig: partdeco}, \ref{fig: arc2p} and \ref{fig: strip2p}. 

Firstly we consider an ideal quadrilateral with two decorated spikes. There are two possible configurations, which are given by the two rows in Figure \ref{fig: partdeco}. In both the cases there is only one horoball connection, named $\be$, coloured in red. From Theorem \ref{main2}, we get that the arc complex of this polygon is PL-homeomorphic to a closed ball of dimension 2. The arc complexes for both the configurations are shown in the middle panel of Fig. \ref{fig: partdeco}. 

The images of the projectivised strip map are shown in the right panel. In this affine chart, the admissible cone covers the entire plane. Its boundary is the circle at infinity given by the positive projectivisation of the hyperplane $\{\mathrm d l_\be =0\}\subset \tang \Pi$, containing all the infinitesimal deformations that do not lengthen the horoball connecction $\be$ but gives motion to the undecorated spikes.

In the second example we have a once-punctured bigon with both its spikes decorated. There are four horoball connections denoted by $\ga_1, \ldots, \ga_4$. The arc complex is again a two dimensional closed ball. The boundary of the admissible cone is formed by the four hyperplanes, given by $\{ \mathrm d l_{\ga_i} =0\}\subset \tang \Pi \}$, for $i=1,2,3,4$. 

\begin{proof}
Firstly, we show that the image of $\mathbb{P}f$ is contained in $\mathbb{P}^+(\overline{\adm m)}$.
Recall that the set $\overline{\adm m}$ is the intersection of finitely many hyperspaces of the form $$\{ u\in \tang \Pi \mid \mathrm{d} l_\be (m)(u) \geq 0\},$$ where $\be$ is a horoball connection of $\Pi$.
From Theorem \ref{cor: pardeco}, we know that $\mathbb{P}f(\sac{\Pi})=\mathbb P^+(\adm m)$. 
Now let $x\in \hac \Pi \setminus \sac \Pi$. Then $x$ is contained in the interior of a unique boundary simplex $\sigma_x$. Then the arc family corresponding to $\sigma_x$ is disjoint from some horoball connection $\be_0$. 
We know from \cite{panda2023strip} that $$\mathrm{d} l_\be (m)(f(x))\left\{ \begin{array}{rl}
	>0, & \text{when $x$ intersects $\be$},\\
	=0, & \text{otherwise}.
	\end{array}
	\right. $$
So $\mathbb{P}f(x)\in \mathbb{P}^+ (\{\mathrm{d}l_{\be_0}=0\}\cap\partial\overline{\adm m})$.

It suffices to prove the homeomorphism in the deprojectivised case: 
\[ f: C(\hac \Pi) \longrightarrow \overline{\adm m},\]
where $C(\hac \Pi)$ is homeomorphic to $\R_{>0}\times\ball {N-1}$, using Theorems \ref{main2}, \ref{mainp}. Here $N$ is the dimension of $\tang \Pi$. The admissible cone $\overline{\adm m}$ is a also closed convex cone of $\tang \Pi \setminus \{0\}.$ In the following we prove that $f$ is a local homeomorphism and a proper map. This implies that $f$ is a covering map. Since both these manifolds $C(\hac\Pi)$ and $\overline{\adm m}$ are simply connected and have the same dimension, we can then conclude that $f$ is a homeomorphism.  

Firstly, we show that the deprojectivised strip map is proper. Let $\{x_n\}_n \in C(\hac\Pi)$ be a sequence such that $x_n\to\infty$. Since the arc complex $\hac\Pi$ is a finite simplicial complex, we can assume (up to extracting a subsequence) that $x_n\in C(\sigma)$, where $\sigma \in \hac\Pi$ is a simplex. Then we can write $x_n=\sum\limits_{i}c_i ^{n}\al_i$ where $\al_i\in \sigma^{(0)}$,  $c_i^{n}\geq 0$ and $w(x_n):=\sum\limits_{i} c_i^{n} >0$. Then, the hypothesis $x_n\to\infty$ implies that, up to a subsequence, either $w(x_n)\to 0 $ or $w(x_n)\to\infty$. We need to prove that $f(x_n)\to\infty$ in $\overline{\adm m}$. In other words, we need to show that the norm of $f(x_n)$ tends to either 0 or to~$\infty$. From Theorem \ref{oldthm: basis}, we know that the set of infinitesimal strip deformations $\{f_{\al_i}(m)\mid \al_i \in \sigma^{(0)}\}$ is linearly independent. So there exists $k,K>0$ such that $kw(x_n)\leq f(x_n)\leq Kw(x_n)$. This proves that $f(x_n)\to \infty$.

Now we prove that the deprojectivised strip map $f$ is a local homeomorphism. Given a point $x\in C(\hac \Pi)$, denote by $[x]$ its image in $\hac \Pi$. Suppose that $x$ is contained in the interior of $C(\sigma_{[x]})$, for a simplex $\sigma_{[x]}$ of $\hac \Pi$ containing $[x]$. It suffices to show that this map $f$ is a homeomorphism for all $d:=\codim{\sigma_{[x]}}\geq 0$. 
Suppose that $d=0$. Then from Theorem \ref{oldthm: basis}, we have that $\{f_\al(m)\mid \al \in \sigma_{[x]}^{(0)}\}$ is a basis of $\tang \Pi$. So $f$ is a linear homeomorphism in the interior of $C(\sigma_{[x]})$.
Next suppose that $d=1$. The case when $\sigma_{[x]}$ is a internal simplex of $\hac\Pi$ follows from Theorem \ref{thm: oldppdeco}. Suppose that $\sigma_{[x]}$ is a boundary simplex. Then it is contained in exactly one face of codimension 0, say $\Delta$, and $f(\sigma_{[x]})\in \partial\adm m$. Furthermore, from the codimension 0 case we know that $f(C(\Delta))$ is non-degenerate. Hence, $f(C(\sigma_{[x]}))$ is a codimension one boundary facet of $\overline{\adm m}$.

Now suppose that the deprojectivised strip map is a local homeomorphism around all  points $x\in C(\hac \Pi)$ such that $\codim {\sigma_{[x]}} =0, 1, \ldots, d-1$. We already know that $f$ is a local homeomorphism when $x$ is in the interior of $C(\hac\Pi)$. So we need to treat only the case when $x\in \partial\hac\Pi$ such that $\codim {\sigma_{[x]}} =d$. We can write $x=\sum\limits_{\al\in \sigma_{[x]}^{(0)}} c_{\al}\al$, where $c_\al >0$. Let $y$ be a point close to $x$ in $C(\hac\Pi)$. Let $y':=y-x=\sum\limits_{\al\in \sigma_{[x]}^{(0)}} c'_{\al}\al + \sum\limits_{\be\in \mathrm{Lk}^{(0)}}c'_{\be}\be$, where $c'_{\al}\geq 0$. Since $y=x+y'\in C(\hac\Pi)$, we have that $c'_\be\in \R$. Define $g(y'):=f(x+y')-f(x)$ and extend it to the full tangent cone of $C\hac\Pi$ at $x$ to get $$g:C\Link{\sigma_{[x]}}{\hac\Pi} \times \R^{N-d} \longrightarrow \mathrm{T}_{f(x)}\tang \Pi.$$ Now for $y':=\sum\limits_{\al\in \sigma_{[x]}^{(0)}} c'_{\al}\al \in C(\sigma_{[x]})$, $y'':= \sum\limits_{\al\in \sigma_{[x]}^{(0)}} c''_{\al}\al+ \sum\limits_{\be\in \mathrm{Lk}^{(0)}}c''_{\be}\be \in C\hac\Pi$, we have that 
\[
g(y')+g(y'')=\sum\limits_{\al\in \sigma_{[x]}^{(0)}} (c'_{\al}+c''_{\al})\al+ \sum\limits_{\be\in \mathrm{Lk}^{(0)}}c'_{\be}\be=g(y'+y'').
\]
Furthermore, we have that for $t>0$, $g(ty')=tg(y)$. So the local homeomorphism of $f$ at $x$ is equivalent to homeomorphism of the map given by \[    \overline{g}:\Link{\sigma_{[x]}}{\hac\Pi} \rightarrow \mathbb P^+(\tang \Pi /\Span{f(\sigma_{[x]})}),\]
From Theorems \ref{main2} and \ref{mainp}, we know that $\Link{\sigma_{[x]}}{\hac\Pi})\simeq \ball{d-1}$. So $\overline{g}$ becomes a map defined between two spaces both homeomorphic to $\ball{d-1}$. 
Let $y\in \Link{\sigma_{[x]}}{\hac\Pi}$, contained in the interior of the simplex $\sigma_y$ of the link. Then we have 
\[\Link{\sigma_{y}}{\Link{\sigma_{[x]}}{\hac\Pi}}= \Link{\sigma_{[x]}\cup \sigma_{y}}{\hac\Pi}=\left\{
\begin{array}{rl}
\ball{d'-1}, & \text{ if } \sigma_{[x]}\cup \sigma_{y} \in \partial \hac\Pi,\\
\s{d'-1}, & \text{ otherwise,}
\end{array}
\right.\]
where $d':=\codim {\sigma_{[x]}\cup \sigma_{y}}<d=\codim{\sigma_{[x]}}$. By induction hypothesis, $\overline{g}$ is a local homeomorphism around $y$. So $\overline{g}$ is a local homeomorphism between two simply connected spaces, hence a global homeomorphism. This concludes the proof of local homeomorphism of $f$.




\end{proof}

\newpage
\bibliography{colouredpolygons.bib}
\bibliographystyle{plain}
\end{document}